\documentclass[smallextended]{svjour3}      

\smartqed

\usepackage{graphicx}
\usepackage{amsfonts}
\usepackage[linesnumbered,ruled,vlined]{algorithm2e}
\usepackage{amsmath}
\usepackage{multirow}
\usepackage{subfig}
\usepackage{tikz}
\usepackage{cite}
\usepackage{enumerate}

\journalname{}

\begin{document}

\title{Derivative-Free Superiorization With Component-Wise Perturbations
	\thanks{This project was supported by Research Grant No. 2013003 of the United 	States-Israel Binational Science Foundation (BSF) and by Award No. 	1P20183640-01A1 of the National Cancer Institute (NCI) of the National 	Institutes of Health (NIH).}
}


\titlerunning{Derivative-Free Superiorization} 

\author{Yair Censor \and Howard Heaton \and Reinhard Schulte}
\authorrunning{Y. Censor, H. Heaton, R. Schulte} 

\institute{Yair Censor \at
           Department of Mathematics,
           University of Haifa,
           Mt. Carmel, Haifa, 3498838, Israel \\
          \email{yair@math.haifa.ac.il}
          \and
          Howard Heaton \at
          Department of Mathematics,
          University of California Los Angeles,
          Los Angeles, CA, 90095, USA
          \and
          Reinhard Schulte \at
          Division of Radiation Research,
          Department of Basic Sciences,
          School of Medicine,
          Loma Linda University,
          Loma Linda, CA 92350, USA
}

\date{Received: date / Accepted: date /
\textbf{Original submission: August 31, 2017. Revised: January 16, 2018. Revised: March 28, 2018.}}

\maketitle

\begin{abstract}
Superiorization reduces, not necessarily minimizes, the value of a target function while seeking constraints-compatibility. This is done by taking a solely feasibility-seeking algorithm, analyzing its perturbations resilience, and proactively perturbing its iterates accordingly to steer them toward a feasible point with reduced value of the target function. When the perturbation steps are computationally efficient, this enables generation of a superior result with essentially the same computational cost as that of the original feasibility-seeking algorithm. In this work, we refine previous formulations of the superiorization method to create a more general framework, enabling target function reduction steps that do not require partial derivatives of the target function. In perturbations that use partial derivatives the step-sizes in the perturbation phase of the superiorization method are chosen independently from the choice of the nonascent directions. This is no longer true when component-wise perturbations are employed. In that case, the step-sizes must be linked to the choice of the nonascent direction in every step. Besides presenting and validating these notions, we give a computational demonstration of superiorization with component-wise perturbations for a problem of computerized tomography image reconstruction.

\end{abstract}

\section{Introduction}
\label{intro}
In this introduction, we first describe briefly the superiorization methodology and mention some of the previous work on it. Then we describe in general terms the proposed approach of derivative-free component-wise perturbations and the computational demonstration that we report about here.
\paragraph{\textbf{The superiorization methodology}.}
The superiorization methodology (SM) is an algorithmic scheme that can be considered to reside between feasibility-seeking and constrained minimization. Rather than attempting to solve a full-fledged minimization problem, the SM takes a feasibility-seeking algorithm and proactively steers its iterates to find a feasible point that is superior, though not necessarily optimal, with respect to the value of a target function, to the output obtained by the feasibility-seeking algorithm. This approach originates from the discovery that many feasibility-seeking algorithms are perturbation resilient in the sense that, even if certain kinds of changes are made at the end of each iterative step, the algorithms still produce constraints-compatible solutions \cite{butnariu2007stable,censor2010perturbation,davidi2009perturbation,nikazad2012accelerated}.

When the steps to compute perturbations of the iterates of a feasibility-seeking algorithm to reduce the target function value are computationally efficient, a superior result is obtained with essentially the same computational cost as that of the original feasibility-seeking algorithm. Thus, the SM is useful for constrained minimization problems where either an exact algorithm has not been discovered or existing exact algorithms are exceedingly time consuming or require too much computer space for realistically large problems to be solved on commonplace computers. In these cases, the SM enables efficient feasibility-seeking algorithms, which provide constraints-compatible solutions, to be turned into efficient algorithms that will be practically useful from the point of view of reducing the value of the underlying target function.

\paragraph{\textbf{Previous work on superiorization}.}
In the SM, the superiorized version of an iterative feasibility-seeking algorithm consists of two parts. The first part performs perturbations that aim to reduce the value of the target function. The other is a part where the operator for the feasibility-seeking algorithm is applied. As noted in \cite{censor2014projectedsubgrad}, several works have made use of this idea with proposed algorithms for exact constrained minimization (e.g., see \cite{neto2009incremental,neto2011perturbed,defrise2011algorithm,sidky2011constrained,sidky2008image,bian2010evaluation,nurminski2010envelope,combettes2002adaptive,combettes2004image}). However, these approaches are unable to do what is accomplished by the superiorization approach, which is to automatically generate a heuristic constrained optimization algorithm from an iterative feasibility-seeking algorithm. The underlying idea of the SM is quite general and provides application in many areas. The mathematical principles of the SM over general consistent ``problem structures'' with the notion of bounded perturbation resilience were formulated in \cite{censor2010perturbation}. The framework of the SM was extended to the inconsistent case by using the notion of strong perturbation resilience in \cite{censor2014projectedsubgrad,censor2015weakstrong}. In \cite{censor2014projectedsubgrad}, the efficacy of the SM was also demonstrated by comparing it with the performance of the projected subgradient method for constrained minimization problems.

A comprehensive overview of the state of the art and current research on superiorization appears in our continuously updated bibliography Internet page which currently contains 76 items \cite{sup-bib}. Research works in this bibliography include a variety of reports ranging from new applications to new mathematical results of the foundation of superiorization. A special issue entitled: \textquotedblleft Superiorization: Theory and Applications\textquotedblright{} of the journal Inverse Problems, has recently appeared \cite{Sup-Special-Issue-2017}.

\paragraph{\textbf{Derivative-free component-wise perturbations}.}
In the SM, the perturbation part interlaces target function reduction steps into the feasibility-seeking algorithm. Until now, generation of nonascent directions, used for target function reduction steps, was mostly based on theorems such as \cite[Theorem 1]{herman2012superiorization} and its variants such as \cite[Theorem 1]{garduno2014} and \cite[unnumbered Theorem on page 7]{garduno2017}. All these theorems make the assumption on the constructed nonascent direction $g$ whose existence is guaranteed by the theorem that: ``Let $g\in \mathbb{R}^{J}$ satisfy the property: For $1\leq j\leq J$, if the $j$th component $g_{j}$ of $g$ is not zero, then the partial derivative $\frac{\partial\phi}{\partial x_{j}}(x)$ of $\phi$ at $x$ exists and its value is $g_{j}$.'' Thus, $\phi$ must have at least one partial derivative (which is nonzero) at points in the domain of $\phi$. Otherwise, these theorems would apply only to the zero-vector, which is a useless nonascending vector because it renders the SM ineffective.

To summarize this point, the definition of nonascending vectors (see Definition \ref{def:nonascend} below) does not require differentiability but in almost all existing works $\phi$ must obey the condition to have at least one partial derivative (which is nonzero) at points in its domain.
The paper \cite{garduno2011} is a possible exception since no derivatives are used there, but it refers only to the specific $\ell_{1}$-norm and still does not answer the general question of how to implement the SM in cases when the above mentioned theorems do not apply due to total lack of partial derivatives. This question makes sense for cases in which only target function values can be calculated but nothing else about the function, such as, for example, functions that are defined by tables of values. Many such derivative-free objective functions are available in the field of derivative-free optimization, see, e.g., \cite{rios2013}.

In this paper, we offer an approach how to handle target functions $\phi$ which do not obey the above condition of having at least one nonzero partial derivative or for which one is unable to verify it. Our main contribution is to propose the use of \textit{component-wise perturbations} within the SM. When component-wise perturbations are employed, the classical notion of nonascent does not necessarily apply because if a point at a certain distance from a given point, along some coordinate, has a lower target function value it does not guarantee that any other point in the neighborhood of the given point does so. In such a case, the step-sizes of the perturbations must be properly linked to the choice of the nonascent direction in every step, giving rise to a new notion of ``local nonascent'' of the target function, see Definition \ref{def:nonascent} below). These notions of local nonascent and component-wise perturbations have not been used in superiorization until now and they have both theoretical and practical significance. Such formulation of superiorization is logically more generally applicable than previously studied superiorization methods since it allows a wider selection of target function reduction steps and enables the SM to be applied with target functions for which not even one partial derivative is available.

\paragraph{\textbf{Computational demonstration}.}
By considering component-wise perturbations, we generalize previous superiorization schemes to enable use of a wider selection of methods for step-wise reduction of the target function. As a first step in showing that component-wise perturbations in the SM work, we present a new superiorization scheme for reducing total variation (TV) during image reconstruction, i.e., total variation superiorization (TVS). We decided to do component-wise perturbations on iterates to reduce the TV function although it has calculable partial derivatives. This way, we  have something to compare our results with. Surprisingly, we found that, even for this sub-differentiable target function, component-wise perturbations can outperform negative gradient perturbations within the SM. This is not to say or claim that component-wise perturbations always outperform perturbations based on derivative information. On the contrary, it is expected that gradient-based perturbations will, in general, be more efficient in the SM. The true merit of component-wise perturbations is that it opens of the door for derivative-free perturbations in the SM, e.g., by applying it to superiorization of biological merit functions in intensity-modulated radiation therapy (IMRT). Another computational demonstration of derivative-free perturbations in the SM, based on the ideas presented here that we communicated to the authors, appears in \cite[Section 4.3]{gibalipetra2017}.

\paragraph{\textbf{What is in this paper}.}
The remainder of this work is outlined as follows. In Section \ref{sec:The-superiorization-methodology}, we present the mathematical framework of the SM in the context of solving a convex feasibility problem, which is followed by our proposed scheme for TVS with component-wise perturbations in Section \ref{sec:CD for TVS-1}. Then Section \ref{sec:A-computational-example} provides an example of the specific proposed scheme for TVS applied to image reconstruction, juxtaposing our approach with a negative gradient-based approach based on previous works (e.g., \cite{herman2012superiorization,censor2014projectedsubgrad}). Discussion and conclusions are provided, respectively, in Section \ref{sec: Discussion} and Section \ref{sec:Conclusions}.

\section{Superiorization with Local Nonascent \label{sec:The-superiorization-methodology}}

\subsection{The superiorization framework}
\noindent In order to make the paper to some extent self-contained, we briefly review the SM framework as developed in earlier publications, see, e.g., \cite{censor2015weakstrong,censor2010perturbation,censor2014projectedsubgrad,strictfejer,herman2012superiorization}. Given a collection of closed convex subsets $C_{i}\subset\mathbb{R}^{L}$ for $i=1,2,\ldots,m$, in the $L$-dimensional Euclidean space, the convex feasibility problem (CFP) is to find a point $x^{*}\in\cap_{i=1}^{m}C_{i}$. In the superiorization method, one seeks a solution to the CFP that is superior, although not necessarily optimal, with respect to some target function $\phi$. A superior solution is here considered to be a better solution, with respect to the target function value, than that which would have been found by the given feasibility-seeking algorithm without superiorization steps. Suppose that we have a feasibility-seeking algorithmic operator $\mathcal{A}:\mathbb{R}^{L}\rightarrow\mathbb{R}^{L}$ with which we define an iterative process for the solution of a CFP

\begin{equation}
x^{k+1}=\mathcal{A}(x^{k})\ \ \mbox{for all }k\geq0\ \mbox{with arbitrary}\ x^{0}\in\mathbb{R}^{L}.
\end{equation}

This process is called ``the basic algorithm'' and the sequence of iterates it produces can be evaluated using a notion of proximity to the sets of the CFP. Let $\lbrace C_{i}\rbrace_{i=1}^{m}$ be a finite family of closed convex sets and suppose the existence of a nonempty subset $\Lambda\subset\mathbb{R}^{L}$ such that $C_{i}\subset\Lambda$ for all $i=1,2,\ldots,m,$. We denote this CFP by $T$ and associate with it a \textit{proximity function} $\mbox{Prox}_{T}:\Lambda\rightarrow\mathbb{R}_{+}$ that indicates how compatible an $x\in\Lambda$ is with the constraints. Given any positive $\varepsilon$, any point $x\in\Lambda$ for which $\mbox{Prox}_{T}(x)\leq\varepsilon$ is called an \textit{$\varepsilon$-compatible} solution of the CFP. Thus, the basic algorithm can be terminated when the proximity function gives a value less than some positive $\varepsilon$. We define this as the \textit{$\varepsilon$-output of a sequence} of points generated by an iterative algorithmic operator, see \cite[page 5]{censor2010perturbation}.
\begin{definition}
	Given a family of constraints sets $\lbrace C_{i}\rbrace_{i=1}^{m}$ of a CFP $T$, a proximity function
	$\mbox{Prox}_{T}:\Lambda\rightarrow\mathbb{R}_{+}$, a sequence $\{x^{k}\}_{k=0}^{\infty}\subset\Lambda$
	and an
 $\varepsilon>0$, an element $x^{K}$ of the sequence which has the
	properties:
	\begin{enumerate}[i)]
		\item
			$\mbox{Prox}_{T}(x^{K})\leq\varepsilon,$ and
		\item
			$\mbox{Prox}_{T}(x^{k})>\varepsilon$ for all $0\leq k<K$,
	\end{enumerate}
	is called the \textbf{$\varepsilon$-output of the sequence} $\{x^{k}\}_{k=0}^{\infty}$
	with respect to the pair $(T,\mbox{Prox}_{T})$.
\end{definition}
The $\varepsilon$-output $x^{K}$ of a sequence is denoted by $\mathcal{O}\left(T,\varepsilon,\lbrace x^{k}\rbrace_{k=0}^{\infty}\right)$. Such an output may not exist; however, when it does, it is unique. Furthermore, when the sequence $\lbrace x^{k}\rbrace_{k=0}^{\infty}$ is generated by a basic algorithm for solving a CFP, the point $\mathcal{O}\left(T,\varepsilon,\lbrace x^{k}\rbrace_{k=0}^{\infty}\right)$ gives the output of the basic algorithm when the stopping criterion is $\varepsilon$-compatibility.

The following version of the SM, presented in \cite{herman2012superiorization}, is known as strong superiorization. (See \cite{censor2015weakstrong} for a review of strong and weak superiorization.) Here the solution set $C$ of the CFP $T$ may be empty and solving the CFP is then understood to mean finding a point that is within a given proximity of the constraints. The ``superiorized version of a basic algorithm'' is created by taking advantage of the fact that successive iterates of the basic algorithm can, in some instances, be systematically perturbed without losing overall convergence of the iterates. Our problem at hand is stated as follows.
\begin{problem}
	
Let $\lbrace C_{i}\rbrace_{i=1}^{m}$ be a family of closed convex sets of a CFP $T$, $C_{i}\subseteq\Lambda\subseteq\mathbb{R}^{L}$ for all $i$, let $\phi:\mathbb{R}^{L}\rightarrow\mathbb{R}$ be a given target function and let $\mathcal{A}:\Lambda\rightarrow\mathbb{R}^{L}$ be an iterative algorithmic operator defining a basic algorithm for solving the associated CFP. The \textbf{function reduction problem} is to use a superiorized version of the basic algorithm to find a point $x^{*}$ that is $\varepsilon$-compatible with $C$ and has a lesser value of the function $\phi$ than that of another $\varepsilon$-compatible point that would have been obtained by applying the basic algorithm alone.
\end{problem}
Strong perturbation resilience is a property that describes the ability of a basic algorithm to be perturbed and not lose its ability to yield an \textit{$\varepsilon$-compatible} solution of the CFP. This notion was termed ``bounded perturbation resilience'' in \cite[Subsection II.C]{herman2012superiorization} and is defined as follows.
\begin{definition}
	\label{def: strongly-pert-resilient}Assume we are given family of constraints $\lbrace C_{i}\rbrace_{i=1}^{m}$ of a CFP $T$, a proximity function $\mbox{Prox}_{T}$, an algorithmic operator $\mathcal{A}$ and an $x^{0}\in\Lambda$. We use $\{x^{k}\}_{k=0}^{\infty}$ to denote the sequence generated by the basic algorithm when it is initialized at $x^{0}$. The basic algorithm is said to be \textbf{strongly perturbation resilient} iff the following hold:
	
	\begin{enumerate}[i)]
		\item
			 there exist an $\varepsilon>0$ such that the $\varepsilon$-output
			 $\mathcal{O}\left(T,\varepsilon,\{x^{k}\}_{k=0}^{\infty}\right)$
			 exists for every $x^{0}\in\Lambda$; and
		\item			
			for every $\varepsilon>0$, for which the $\varepsilon$-output
			$\mathcal{O}\left(T,\varepsilon,\{x^{k}\}_{k=0}^{\infty}\right)$
			exists for every $x^{0}\in\Lambda$, the $\varepsilon'$-output $\mathcal{O}\left(T,\varepsilon',\{y^{k}\}_{k=0}^{\infty}\right)$
			also exists for every $\varepsilon'>\varepsilon$ and for every sequence
			$\{y^{k}\}_{k=0}^{\infty}$ generated by
			\begin{equation}
			y^{k+1}:=\mathcal{A}\left(y^{k}+\beta_{k}v^{k}\right),\ \ \mbox{for all}\ k\geq0,
			\end{equation}
			where the vector sequence $\{v^{k}\}_{k=0}^{\infty}$ is bounded and
			the scalars $\{\beta_{k}\}_{k=0}^{\infty}$ are such that $\beta_{k}\geq0$
			for all $k\geq0$ and the $\beta_{k}$ are summable, i.e.,
\begin{equation}
\sum_{k=0}^{\infty}\beta_{k}<\infty.
\end{equation}
 
	\end{enumerate}
\end{definition}
Sufficient conditions for strong perturbation resilience of a basic algorithm were proven in \cite[Theorem 1]{herman2012superiorization}.

\subsection{Locally nonascending directions}
The chief motivation to perturb iterates of a basic algorithm by sequences $\{\beta_{k}\}_{k=0}^{\infty}$ and $\{v^{k}\}_{k=0}^{\infty}$ is to reduce the values of the target function $\phi$ by employing directions of nonascent. Below we present the definition of nonascent that is in use in all works on the SM, see \cite[Subsection II.D]{censor2014projectedsubgrad}.
\begin{definition}
	\label{def:nonascend}Given a function $\phi:\mathbb{R}^{L}\rightarrow \mathbb{R}$ and a point $y\in \mathbb{R}^{L}$, we say that a vector $d\in \mathbb{R}^{L}$ is \textbf{nonascending for $\phi$ at $y$} iff $\|d\|\leq1$ ($\|\cdot\|$ denotes the Euclidean norm) and there is a $\delta>0$ such that
	\begin{equation}
	\mbox{for all }\mu\in\left[0,\delta\right]\mbox{ we have }\phi\left(y+\mu d\right)\leq\phi\left(y\right).\label{eq:nonascend}
	\end{equation}
\end{definition}
This definition asserts that the nonascent inequality in (\ref{eq:nonascend}) holds throughout the interval $\left[0,\delta\right].$ Under such circumstances, one can dictate the step-sizes in the perturbation phase of the SM independently of the choice of the nonascent vector. However, in order to employ component-wise perturbations or other perturbations which do not assume the availability of \textbf{any} partial derivative of the target function $\phi$ at $y$, we need to use a different definition of nonascent directions.
We wish to allow the user to look in a neighborhood of the current point $y$ for a point where the target function value is lower without assuming that it is lower in an interval around the current point (this could be the case, e.g., with nonconvex target functions). To do this, the choice of nonascent direction and the perturbation step-size must be linked together to guarantee the reduced target function value. Therefore, we relax the above definition of nonascending vectors so that we may use a wider class of perturbations such as, in particular, component-wise perturbations.
\begin{definition}
	\label{def:nonascent}Given a target function $\phi:\Delta\rightarrow\mathbb{R}$ where $\Delta\subset\mathbb{R}^{L}$, a point $y\in\Delta$, and a positive $\delta\in\mathbb{R}$, we say that $d\in\mathbb{R}^{L}$ is a \textbf{nonascending $\delta$-bound direction for $\phi$ at $y$} if $\|d\|\leq\delta$ and $\phi(y+d)\leq\phi(y)$. The collection of all such vectors is called a \textbf{nonascending $\delta$-ball} and is denoted by $\mathcal{B}_{\delta,\phi}(y)$, i.e.,
	\begin{equation}
	\mathcal{B}_{\delta,\phi}(y):=\{d\in\mathbb{R}^{L}\mid\|d\|\leq\delta,\ \phi(y+d)\leq\phi(y)\}.
	\end{equation}
\end{definition}
	The zero vector is contained in each nonascending $\delta$-ball, i.e., $0\in\mathcal{B}_{\delta,\phi}(y)$ for each $\delta>0$ and $y\in\Delta$. This definition will allow us to use as a nonascent direction any vector $d\in \mathbb{R}^{L}$ at which $\phi(y+d)\leq\phi(y)$ holds, which might be detected by only function value calculations. This will be useful even when $\phi$ is not convex, or if we do component-wise search for a point with reduced target function value. Even functions defined by tabular representations are valid candidates for this nonascending $\delta$-bound directions. We refer to this kind of nonascent as ``local nonascent''.

\subsection{Superiorized version of a basic algorithm with locally nonascending
	directions\label{subsec:Superiorized-version-of}}

The superiorized version of the basic algorithm presented here in Algorithm \ref{alg: sup-version} assumes that we have a summable sequence $\lbrace\eta_{\ell}\rbrace_{\ell=0}^{\infty}$ of positive real numbers generated by $\eta_{\ell}:=a^{\ell}$ where $a\in(0,1)$, called kernel in \cite{censorLinSup}, is user-chosen. This summable sequence is used to perturb iterates with the goal to reduce the value of the target function $\phi$ while maintaining convergence of the iterates to a solution of the original CFP. Each $\eta_{\ell}$ is used to generate a nonascending $\eta_{\ell}$-ball for $\phi$ about iterates produced by applying the basic algorithmic operator $\mathcal{A}$. Points chosen from each of these $\eta_{\ell}$-balls generate sequences $\lbrace v^{k}\rbrace_{k=0}^{\infty}$ and $\lbrace\beta_{k}\rbrace_{k=0}^{\infty}$, corresponding to the sequences in Definition \ref{def: strongly-pert-resilient}. These sequences aim to steer the sequence to a lesser value of $\phi$. This superiorized version of the basic algorithm also depends on a chosen initial point $\bar{y}$ and a sequence $\lbrace N_{k}\rbrace_{k=0}^{\infty}$ of positive integers bounded by some positive integer $N$. With this, the superiorized version of the basic algorithm is presented in Algorithm \ref{alg: sup-version} by its pseudo-code.

\begin{algorithm}
	\caption{Superiorized version using local nonascent of a strongly perturbation resilient basic algorithm \label{alg: sup-version}}
	\normalsize
		$k\leftarrow0$
		\label{step: k=00003D0} \par
		$y^{k}\leftarrow\bar{y}$
		\label{step: y0}\par
		$\ell\leftarrow0$
		\label{step: L0} \par
		\textbf{while }stopping criterion not met
		\label{step: lp_st}\par
		\hspace{0.75cm}$y^{k,0}\leftarrow y^{k}$
		\label{step: yk1} \par
		\hspace{0.75cm}\textbf{for} $n=0\mbox{ to }(N_{k}-1)$ \textbf{do}
		\label{step: for_st}\par
		\hspace{1.5cm}
		Let $v^{k,n}\in\mathcal{B}_{\eta_{\ell},\phi}(y^{k,n})$
		$\ \ \ $(see Definition \ref{def:nonascent})
		\label{step: pick_v}\par
		\hspace{1.5cm}
		$y^{k,n+1}\leftarrow y^{k,n}+v^{k,n}$
		\label{step:8}\par
		\hspace{1.5cm}
		$\ell\leftarrow\ell+1$
		\label{step: L++}\par
		\hspace{0.75cm}
		\textbf{end for}
		\label{step: for_end} \par
		\hspace{0.75cm}
		$y^{k+1}\leftarrow\mathcal{A}\left(y^{k,N_{k}}\right)$
		\label{step: A}	\par
		\hspace{0.75cm}
		$k\leftarrow k+1$
		\par
		\textbf{end while}
		\label{step: lp_end}\par
	\bigskip
\end{algorithm}

The behavior of this superiorized version of a basic algorithm is analyzed here according to how well it achieves feasibility and according to how well it reduces the target function values. For the feasibility question, we have the following lemma, which resembles the arguments in \cite[Subection II.E]{herman2012superiorization} but differs in the use of local nonascending directions.
\begin{lemma}
	\label{lem:epsilon-prime}Assume that a basic algorithm represented 	by the algorithmic operator $\mathcal{A}$ is strongly perturbation resilient and produces an $\varepsilon$-compatible output for some $\varepsilon>0$. If $\lbrace\eta_{\ell}\rbrace_{\ell=0}^{\infty}$ is a summable sequence of positive real numbers, then the superiorized version of the basic algorithm using local nonascent, given by Algorithm \ref{alg: sup-version}, produces an $\varepsilon'$-compatible output for each $\varepsilon'>\varepsilon$.
\end{lemma}
\begin{proof}
	We have to show that if $\mathcal{O}\left(T,\varepsilon,\{y^{k}\}_{k=0}^{\infty}\right)$  is defined for each $y^{0}\in\mathbb{R}^{L}$, then for any $\varepsilon'>\varepsilon$, Algorithm \ref{alg: sup-version} produces an $\varepsilon'$-compatible output. The strong perturbation resilience of $\mathcal{A}$ guarantees this if there exist a summable sequence $\{\beta_{k}\}_{k=0}^{\infty}$ of nonnegative real numbers and a bounded sequence $\{v^{k}\}_{k=0}^{\infty}$ of vectors in $\mathbb{R}^{L}$ such that 	
	\begin{equation}
	y^{k+1}=\mathcal{A}(y^{k}+\beta_{k}v^{k})\ \ \ \forall\ k\geq0.
	\end{equation}
	Indeed, define 	
	\begin{equation}
	\beta_{k}:=\max\{\|v^{k,n}\|\mid0\leq n\leq N_{k}-1\}\label{eq: beta_k}
	\end{equation}
	and	
	\begin{equation}
	v^{k}:=\begin{cases}
	\sum_{n=0}^{N_{k}-1}\dfrac{1}{\beta_{k}}v^{k,n}, & \mbox{if }\beta_{k}>0,\\
	0, & \mbox{otherwise.}
	\end{cases}\label{eq: v_k}
	\end{equation}
	Since $y^{k,0}=y^{k}$, it follows from Steps \ref{step: yk1}-\ref{step: for_end} that these definitions result in $y^{k,N_{k}}=y^{k}+\beta_{k}v^{k}.$
	From Step \ref{step: pick_v} it follows that $\{\beta_{k}\}_{k=0}^{\infty}$ 	is a subsequence of $\{\eta_{\ell}\}_{\ell=0}^{\infty}$ and, hence, it is a summable sequence of nonnegative real numbers. Because the sequence $\{\eta_{\ell}\}_{\ell=0}^{\infty}$ is summable and each $\|v^{k,n}\|\leq\eta_{\ell}$, for appropriate $\ell$, it follows that $\{v^{k}\}_{k=0}^{\infty}$ is bounded. Hence the superiorized version using local nonascent, given by Algorithm \ref{alg: sup-version}, produces an $\varepsilon'$-compatible output for each $\varepsilon'>\varepsilon$. \qed
\end{proof}
Algorithm \ref{alg: sup-version} works as follows. Initially, the iteration number $k$ is set to 0 and $y^{0}$ is set to its initial value $\bar{y}$. The index $\ell$ for picking the next term of the sequence $\{\eta_{\ell}\}_{\ell=0}^{\infty}$ is initialized to $\ell=0$ and is repeatedly incremented by Step \ref{step: L++}. Steps \ref{step: lp_st}-\ref{step: lp_end} do a full iterative step, from $y^{k}$ to $y^{k+1},$ and repetitions of these steps generate the sequence $\{y^{k}\}_{k=0}^{\infty}$. During one iterative step, there is one application of the operator $\mathcal{A}$, in Step \ref{step: A}, but there are $N_{k}$ steering steps aimed at reducing the value of $\phi$; the latter are done by Steps \ref{step: for_st}-\ref{step: for_end}. These steps produce a sequence of inner loop points $y^{k,n}$, where $0\leq n\leq N_{k}$ with $y^{k,0}=y^{k}$ and $y^{k,n}\in\mathbb{R}^{L}$.

To our knowledge, except for \cite{strictfejer}, no proof has been published to date asserting the precise behavior of a superiorized version of an algorithm regarding  the target function values. However, here we claim that Algorithm \ref{alg: sup-version} systematically reduces target function values within the inner loops of perturbations, similarly to the analysis in \cite[Subection II.E]{herman2012superiorization}.
\begin{theorem}
	Under the conditions of Lemma \ref{lem:epsilon-prime}, sequences of inner loop points $y^{k,n}$, generated by Algorithm \ref{alg: sup-version}, where $0\leq n\leq N_{k}$ with $y^{k,0}=y^{k}$ and $y^{k,n}\in\mathbb{R}^{L}$, have the property that for all $k=0,1,2,\ldots,$ and all $0\leq n\leq N_{k}$,
\end{theorem}
\begin{equation}
\phi(y^{k,n})\leq\phi(y^{k}).
\end{equation}

\begin{proof}
	The proof is by induction. Fix an integer $k\geq0.$ For $n=0$ we have $y^{k,0}=y^{k}$ and so $\phi(y^{k,0})=\phi(y^{k})$. Now assume, for any $0\leq n<N_{k},$ that $\phi(y^{k,n})\leq\phi(y^{k}).$ Next	we show that Steps \ref{step: for_st}-\ref{step: for_end} lead from $y^{k,n}$ to $y^{k,n+1}$ that gives $\phi(y^{k,n+1})\leq\phi(y^{k})$. The vector $v^{n,k}$ in Step \ref{step: pick_v} is chosen, by Definition \ref{def:nonascent}, such that $\phi(y^{k,n}+v^{k,n})\leq\phi(y^{k,n})$. But, in Step \ref{step:8}, $y^{k,n+1}=y^{k,n}+v^{k,n}$ and, by the induction hypothesis, $\phi(y^{k,n})\leq\phi(y^{k})$. Thus,	
	\begin{equation}
	\phi(y^{k,n+1})=\phi(y^{k,n}+v^{k,n})\leq\phi(y^{k,n})\leq\phi(y^{k}).
	\end{equation}
	Therefore, we conclude that $\phi\left(y^{k,n}\right)\leq\phi\left(y^{k}\right)$
	for all $0\leq n\leq N_{k}$.\qed
\end{proof}
After going through the inner loop $N_{k}$ times, Step \ref{step: A} is executed to produce $y^{k+1}$. Then, increasing the value of $k$ allows us to move to the next iterative step. Infinitely many repetitions of such steps produces the sequence of points $\{y^{k}\}_{k=0}^{\infty}$. Due to the repeated steering, by Steps \ref{step: for_st}-\ref{step: for_end}, toward reducing the value of the target function $\phi$, we can expect that the output of the superiorized version using local nonascent will be superior, from the point of view of $\phi$, to the output that would have been obtained, with everything else being equal by the basic algorithm. This ``expected'' outcome has been observed in all published experimental reports to date, see, e.g., the many papers mentioned in \cite{sup-bib}, but has not been yet mathematically proven. On this theoretical side, there is, as far as we know, only the result of \cite{strictfejer}.

\section{Total Variation Superiorization with Component-Wise Perturbations\label{sec:CD for TVS-1}}

\subsection{The application, the approach, and the numerical demonstration}

Total variation (TV) superiorization (TVS) has been used before in image reconstruction from projections with very good experimental performance, as can be seen in several of the papers posted on \cite{sup-bib}. Since TV has everywhere a subgradient, all previous work on TVS used negative subgradients of TV as nonascent directions for the perturbations in the superiorized version of the basic feasibility-seeking algorithm.

In situations of superiorization in the SM with respect to other target functions for which there is no guarantee to have at least one non-zero partial derivative at points in the domain of the function, the notion of $\delta$-bound nonascending perturbations, developed above, plays an important role. As mentioned before, such situations will arise when attempting to apply the SM to target functions $\phi$ which are not convex, or to functions defined by tabular representations.

The purpose of the numerical demonstration presented in the sequel is to show that superiorization with component-wise perturbations works at all. We do not present a full-fledged methodological numerical investigation and, therefore, the findings do not allow to draw general conclusions yet. It would be interesting to see future results when using a larger sample of datasets (e.g., randomized variations of the phantom) and get more statistical information about how superiorization with component-wise perturbations fares in comparison with gradient-based perturbations in the SM.

To explore the numerical behavior of the SM with component-wise perturbations, we wish to have something to compare it with. Therefore, we apply it to TVS without resorting to calculations of its subgradients and compare the results with those obtained from TVS with negative subgradients as directions of nonasecent.

Our computational work surprisingly shows that even in this case in which the target function lends itself to gradient or subgradient calculations, such as TV, component-wise perturbations may be advantageous. Obviously, we do not make any general claim to this effect since more work is needed to investigate the numerical behavior of component-wise perturbations in the SM.

\subsection{Image representation}

Series expansion methods in image reconstruction from projections, see, e.g., \cite{herman2009book}, assume that a two-dimensional (2D) image can be represented using a linear combination of a set of fixed basis functions. Let $f:\mathbb{R}^{2}\rightarrow\mathbb{R}$ be a 2D image. Then a digital approximation of $f$ is defined at each point $r\in\mathbb{R}^{2}$ by
\begin{equation}
f(r)\approx\sum_{\ell=1}^{L}u_{\ell}\cdot b_{\ell}(r),\label{eq: image-series-rep}
\end{equation}
where $b_{\ell}$ denotes the $\ell$-th basis function of some finite set $\lbrace b_{\ell}\rbrace_{\ell=1}^{L}$ of appropriately chosen basis functions and each component $u_{\ell}$ of the vector $u\in\mathbb{R}^{L}$ gives a weighting factor for the contribution of $b_{\ell}$. For a given set of basis functions, the image estimate in (\ref{eq: image-series-rep}) is uniquely determined by $u$, which is called the \textit{image vector}.

Pixels form the set of basis functions used in this work. These are picture elements that cover the entire image. Each pixel has the support of a square and is defined by
\begin{equation}
b_{\ell}(r):=\left\{ \begin{array}{cl}
1, & \mbox{if \ensuremath{r} is inside the \ensuremath{\ell}-th pixel},\\
0, & \mbox{otherwise.}
\end{array}\right.
\end{equation}
When using pixel basis functions, each $u_{\ell}$ in (\ref{eq: image-series-rep}) gives the average value of the image $f$ inside the $\ell$-th pixel. Hereafter, we denote the image approximation in (\ref{eq: image-series-rep}) simply by $u$ and use double-indexing $u_{i,j}$ for$\:i,j=1,2,\ldots,J,$ to denote the value of the digital approximation in (\ref{eq: image-series-rep}) at the pixel location $(i,j)$ where the support of $u$ is composed of $L=J^{2}$ pixels.

\subsection{Total variation in imaging\label{subsec:TV}}

The introduction of noise in reconstructed images is inevitable in practice. However, as introduced in \cite{rudin1992nonlinear}, image restoration based on total variation has proven quite effective for a wide range of applications, including inpainting \cite{shen2002mathematical}, super-resolution \cite{Marquina2008image}, image restoration \cite{Beck2009FastTVDenoising,wang2008new,Chan05recentdevelopments}, and medical imaging \cite{han2013image,Needell2013StableImage,sidky2008image}. TV is formally defined as follows.
\begin{definition}
	Let $u:\mathbb{R}^{2}\rightarrow\mathbb{R}$ be a smooth image. Then 	the\textbf{ total variation} of $u$ is defined by 	
	\begin{equation}
	\mbox{TV}(u):=\int\|\nabla u\|,\label{eq: TV-iso}
	\end{equation}
	where $\nabla u$ denotes the gradient\footnote{This is not to be confused with the notion of gradient of a function. The meaning of $\nabla$ will always be understood according to what it operates on.} of $u,$ so that $\|\nabla u\|:=\sqrt{(D_{x}u)^{2}+(D_{y}u)^{2}}$ where $D_{x}$ and $D_{y}$ denote the horizontal and vertical partial derivative operators.
\end{definition}
In the discrete case, the integral in (\ref{eq: TV-iso}) is replaced by a summation over the extent of the pixels of the digital approximation
$u$, so that
\begin{equation}
\mbox{TV}(u)=\sum_{i,j}\sqrt{\left(D_{x}u_{i,j}\right)^{2}+\left(D_{y}u_{i,j}\right)^{2}},
\end{equation}
where the discrete differential operators are given by
\begin{equation}
D_{x}u_{i,j}:=\begin{cases}
u_{i+1,j}-u_{i,j}, & \mbox{if }1\leq i<J,\\
0, & \mbox{otherwise,}
\end{cases}\label{eq: D_x-discrete}
\end{equation}
and
\begin{equation}
D_{y}u_{i,j}:=\begin{cases}
u_{i,j+1}-u_{i,j}, & \mbox{if }1\leq j<J,\\
0, & \mbox{otherwise.}
\end{cases}\label{eq: D_y-discrete}
\end{equation}

\subsection{TVS with component-wise perturbations\label{subsec:TVS-with-Coordinate-wise}}

For TVS, we propose a new algorithm inspired by the framework presented in the previous sections. Our algorithm computes each nonascent vector $v^{k,n}$ for the target function $\phi=\mbox{TV}$, in Step \ref{step: pick_v} of Algorithm \ref{alg: sup-version}, in a specific manner applicable to TVS. Our approach proposes reducing TV by smoothing out local extrema, i.e., reducing their relative magnitude, through a type of averaging. This is accomplished using a first order approximation of nearby points in an image $u$. Recall that for points $r,h\in\mathbb{R}^{2}$
\begin{equation}
u(r)\approx u(r+h)-\left\langle \nabla u(r+h),h\right\rangle
\end{equation}
gives a first-order approximation where $\left\langle \cdot,\cdot\right\rangle $ denotes the scalar product. Smoothing can be accomplished by using averaged first-order approximations in opposite directions from each point, i.e., by the approximation of the average given by
\begin{equation}
\dfrac{1}{2}\left[u(r+h)+u(r-h)\right]\approx u(r)+\dfrac{1}{2}\left[\left\langle \nabla u(r+h),h\right\rangle -\left\langle \nabla u(r-h),h\right\rangle \right].
\end{equation}
The corresponding perturbation of the image $u$, denoted by $w$, is defined, for each $r\in\mathbb{R}^{2},$ so that
\begin{equation}
w(r)=\dfrac{1}{2}\left(\left\langle \nabla u(r+h),h\right\rangle -\left\langle \nabla u(r-h),h\right\rangle \right).\label{eq: perturbation-continuous}
\end{equation}
By construction, this perturbation $w(r)$ yields an image $u(r)+w(r)$ with lower TV than $u(r)$ for sufficiently small $h$ since, for each $r\in\mathbb{R}^{2}$, $u(r)+w(r)$ will have a value between the values of local extrema of $u$ in proximity of $r$. In the actual computations we compute Step \ref{step: pick_v} of Algorithm \ref{alg: sup-version} in a way that assigns a zero perturbation if the vector obtained was ascending, see equation (\ref{eq: pert-w}) below. In the discrete case, when $h$ is taken to be a unit vector along the $x$ axis, the value $w_{i,j}$ of the perturbation $w$ at the pixel location $(i,j)$ is defined so that
\begin{equation}
w_{i,j}=\dfrac{1}{2}\left(D_{x}u_{i,j}-D_{x}u_{i-1,j}\right),
\end{equation}
and when $h$ is taken to be a unit vector along the $y$ axis
\begin{equation}
w_{i,j}=\dfrac{1}{2}\left(D_{y}u_{i,j}-D_{y}u_{i,j-1}\right).
\end{equation}
Note that, due to the indexing of the discrete differential operators in (\ref{eq: D_x-discrete}) and (\ref{eq: D_y-discrete}), the derivatives $D_{x}u_{i,j}$ and $D_{y}u_{i,j}$ are used above instead of $D_{x}u_{i+1,j}$ and $D_{y}u_{i,j+1}$, respectively.

In order to use $w$ to perturb an image $u$ within the SM, we must bound the magnitude of $w$. This can be done by bounding the contribution of each derivative used in the definition of $w$ to compute a perturbation $w^{*}$, which is thereby bounded component-wise. For $\theta>0$, define the operator $\lfloor\cdot\rfloor_{\theta}:\mathbb{R}\rightarrow\mathbb{R}$ by
\begin{equation}
\lfloor\alpha\rfloor_{\theta}:=\min\left\{ \theta,|\alpha|\right\} \cdot\mbox{sgn}(\alpha),
\end{equation}
where $|\cdot|$ denotes absolute value and $\mbox{sgn}$ denotes the signum function. When $h$ is a unit vector along the $x$ axis and $\theta>0$ is given, the value $w_{i,j}^{*}$ of the perturbation vector $w^{*}$ at the pixel location $(i,j)$ is defined to be
\begin{equation}
w_{i,j}^{*}:=\frac{1}{2}\left(\lfloor D_{x}u_{i,j}\rfloor_{\theta}-\lfloor D_{x}u_{i-1,j}\rfloor_{\theta}\right).\label{eq: pert-discrete-x}
\end{equation}
Otherwise, when $h$ gives a unit vector along the $y$ axis we let
\begin{equation}
w_{i,j}^{*}:=\frac{1}{2}\left(\lfloor D_{x}u_{i,j}\rfloor_{\theta}-\lfloor D_{x}u_{i,j-1}\rfloor_{\theta}\right)\label{eq: pert-discrete-y}
\end{equation}
This formulation of $w^{*}$ enables smoothing an image with control of the magnitude of the perturbation by bounding it component-wise. Each $v^{k,n}$ in Step \ref{step: pick_v} of Algorithm \ref{alg: sup-version} can use $w^{*}$ with $\theta=\eta_{\ell}/\sqrt{L}$. We formalize this with the following proposition.
\begin{proposition}
	Let $\delta>0$ be given. If we define $\theta:=\delta/\sqrt{L}$, then the vector $w^{*}$obtained by (\ref{eq: pert-discrete-x}) and/or (\ref{eq: pert-discrete-y}) is such that $\|w^{*}\|\leq\delta$.
\end{proposition}
\begin{proof}
	Let $\delta>0$ be given. Whether $h$ is a unit vector along either the $x$ or $y$ axis, the above definitions allow each $w_{i,j}^{*}$ to equal $w_{i,j}$ while $|w_{i,j}|\leq\theta$. Otherwise, the signs of the terms composing $w_{i,j}^{*}$ will match those of $w_{i,j}$, but with a reduced magnitude so that the relation $|w_{i,j}^{*}|\leq\theta$ always holds. Arranging the pixel values lexicographically to write $w^{*}$ as an image vector, it then follows that
	\begin{equation}
	\|w^{*}\|\leq\underbrace{\sqrt{\theta^{2}+\cdots+\theta^{2}}}_{L\mbox{ terms}}=\theta\cdot\sqrt{L}.
	\end{equation}
	Thus, by choice of $\theta$, $\|w^{*}\|\leq\delta$.\qed
\end{proof}
The new concept of the perturbation vector for TVS defined above allows us to use $w^{*}$ as a $\delta$-bound nonascending vector for $\phi$ at $u$ provided that $\phi(u+w^{*})\leq\phi(u)$ holds. To ensure this, we compute $v^{k,n}$ in Step \ref{step: pick_v} of Algorithm \ref{alg: sup-version} by choosing
\begin{equation}
v^{k,n}:=\begin{cases}
w^{*}, & \mbox{if }\phi(y^{k,n}+w^{*})\leq\phi(y^{k,n}),\\
0, & \mbox{otherwise.}
\end{cases}\label{eq: pert-w}
\end{equation}
In (\ref{eq: pert-w}), we compute $w^{*}$ using either (\ref{eq: pert-discrete-x}) or (\ref{eq: pert-discrete-y}).

\section{TVS with Component-Wise Perturbations Applied to Image Reconstruction\label{sec:A-computational-example}}

\subsection{Image reconstruction problem}

The discretized model in the series expansion approach to the image reconstruction problem of computerized tomography (CT) is described as follows. Some physical entities (e.g., x-rays) are systematically passed through an object to be scanned. Measurements are made of some physical property of these entities (e.g., attenuation). The goal of image reconstructions is to use measurements to reconstruct an image that represents the object scanned as faithfully as possible. Discretizing the object into pixels or voxels and the outer x-rays field into rays, the modeling of CT yields a matrix $A$ called the \textit{system 	matrix} and a corresponding \textit{measurement vector }$y$. For a complete description see, e.g., \cite{herman2009book}. Each measurement $y_{m}$, which is the $m$-th component of the vector $y$, can be approximated by
\begin{equation}
y_{m}\approx\sum_{\ell=1}^{L}u_{\ell}\cdot a_{\ell}^{m},\label{eq: series-def}
\end{equation}
where $a_{\ell}^{m}$ denotes entry of $A$ in the $m$-th row and $\ell$-th column and each $u_{\ell}$ represents the contribution of the $\ell$-th pixel basis function $b_{\ell}$. One commonly used approach to solving the system $Au=y$ is to use a feasibility-seeking projection method, such as an algebraic reconstruction technique (ART) described in the next subsection.

\subsection{Algebraic Reconstruction Techniques}

The basic algorithmic operator $\mathcal{A}$ that we used to solve the image reconstruction problem is the algebraic reconstruction technique (ART) (see \cite[Chapter 11]{herman2009book}). For each row $m$ of the system matrix, denoted by $a^{m}$, we define the operator $T_{m}:\mathbb{R}^{L}\rightarrow\mathbb{R}^{L}$ by
\begin{equation}
T_{m}(u):=u+\lambda\dfrac{y_{m}-\left\langle a^{m},u\right\rangle }{\|a^{m}\|^{2}}a^{m},
\end{equation}
where $\lambda\in(0,2)$ is a relaxation parameter. The basic algorithmic operator $\mathcal{A}:\mathbb{R}^{L}\rightarrow\mathbb{R}^{L}$ is then given by
\begin{equation}
\mathcal{A}(u):=T_{M}\cdots T_{2}T_{1}(u)
\end{equation}
where $M$ denotes the number of rows in the system matrix. From previous works (e.g., \cite{censor2014projectedsubgrad}), it is known that the basic algorithmic operator $\mathcal{A}$ for ART, defined as above, is strongly perturbation resilient.

\subsection{Target function reducing steps\label{subsec:Target-Function-Reducing}}

Two methods were compared in this work: the new component-wise perturbation method for TVS (CW-TVS) and a method using negative gradients for TVS (NG-TVS) based on \cite[pp. 737--738]{censor2014projectedsubgrad}. The target function used in this example was total variation $\phi=\mbox{TV}$. During each perturbation step of the first method, an iterate $y^{k,n}$ was perturbed component-wise using (\ref{eq: pert-discrete-x}) and then (\ref{eq: pert-discrete-y}) using the perturbation size $\frac{1}{2}\eta_{\ell}$ for each. This process was repeated for each perturbation vector $v^{k,n}$ to reduce the value of the target function. The second method that we used was the TVS algorithm called ``Superiorized Version of the Basic Algorithm'' on pages 737--738 in \cite{censor2014projectedsubgrad}. That is, we set
\begin{equation}
v^{k,n}:=-\eta_{\ell}\cdot\dfrac{\nabla\phi\left(y^{k,n}\right)}{\left\Vert \nabla\phi\left(y^{k,n}\right)\right\Vert }\:\:\mbox{ if }\:\:\mbox{TV}\left(y^{k,n}-\eta_{\ell}\cdot\dfrac{\nabla\phi\left(y^{k,n}\right)}{\left\Vert \nabla\phi\left(y^{k,n}\right)\right\Vert }\right)\leq\mbox{TV}\left(y^{k,n}\right),\label{eq: neg-grad-method}
\end{equation}
and if this statement did not hold, then $\ell$ was incremented until the above statement did hold. The computation of the gradient $\nabla\phi\left(y\right)$ of the target function $\phi$ has up to three fraction terms. In the denominator of each fraction is a term of the form
\begin{equation}
\sqrt{\left(D_{x}(y)_{i,j}\right)^{2}+\left(D_{y}(y)_{i,j}\right)^{2}}\label{eq: denom-term}
\end{equation}
at pixel location $(i,j)$. To maintain numerical stability when the expression (\ref{eq: denom-term}) becomes small, we replace the denominator term (\ref{eq: denom-term}) with
\begin{equation}
\gamma_{tol}+\sqrt{\left(D_{x}(y)_{i,j}\right)^{2}+\left(D_{y}(y)_{i,j}\right)^{2}},
\end{equation}
where $\gamma_{tol}:=10^{-12}$.

\subsection{Computational details}

The computations reported here were done with Matlab \cite{MATLAB} on a single machine using a single CPU, a quad core Intel i5-3317U at 1.70 GHz with 4.00 GB RAM. The AIR tools package \cite{Hansen2012airtools} was used to generate the simulated data. All reconstructions were done in the Matlab environment. Differences in reported reconstruction times are, thus, not due to different algorithms implemented in different environments.

Figure \ref{fig: Shepp-Logan-phantom} shows the phantom used in our study, which is a 256 $\times$ 256 digitized version of the Shepp-Logan phantom whose calculated TV is 1461. We used this phantom as created by the AIR tools package \cite{Hansen2012airtools}. It is represented by an image vector with 65,536 components. The values of the components in the Shepp-Logan phantom range from 0 to 1. For our displays, we use the range {[}0,1{]}. Any value below 0 is shown as black and any value above 1 is shown as white and a linear mapping is used in-between. This display window was used for all images presented here.

\begin{figure}
	\centering
	\includegraphics[width=2.5in]{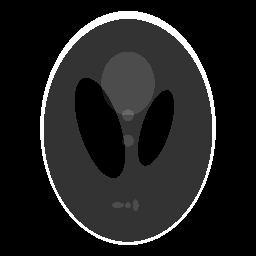}
	\caption{Original $256\times256$ pixel Shepp-Logan phantom with $\mbox{TV}=1461$ \label{fig: Shepp-Logan-phantom}}
\end{figure}

Two sets of experiments were conducted. One had 2\% Gaussian noise added to the measurement data and the other was noise-free. Projection data were collected by approximating line integrals through the digitized phantom in Figure \ref{fig: Shepp-Logan-phantom} using a fan beam, which consists of lines diverging from a single source point.The fan beam was rotated in 15 degree increments about the phantom (24 positions in total) for the noise-free data and in 9 degree increments for the noisy data (40 positions in total). Each line integral gives rise to a linear equation. The phantom itself lies in the intersection of all the solutions of the linear equations associated with these lines. The total number of linear equations generated was 12,288 for the noise-free data and 20,480 for the noisy data, thereby creating an underdetermined problem since there were $256^{2}=65,536$ unknowns. The stopping criterion used for each image reconstruction was when the proximity function
\begin{equation}
\mbox{Prox}_{T}(u):=\|Au-y\|
\end{equation}
yielded a value less than or equal to $\varepsilon=1$ for the noise-free data and $\varepsilon=70$ for the noisy data. The initial iterate for each reconstruction was the zero vector, for which $\mbox{Prox}_{C}(0)=3,497$ in the noise-free case. The specific choice made when running the superiorized version of the basic algorithm for our comparative study were  $\eta_{\ell}=0.2\times0.995^{\ell}$ and $N_{k}=10$ for each $k$. The initial size $\eta_{0}=0.2$ appeared to give the best results for the NG-TVS method when using a kernel $a=0.995$.  The choice of relaxation parameter $\lambda$ when applying ART was $\lambda=1.0$ for the noise-free data and $\lambda=0.2$ for the noisy data.

\begin{figure}
	\centering
	\def\wid{2in}
	\subfloat[]{\includegraphics[width = \wid]{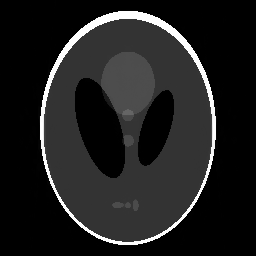}}
	\hspace{0.25in}
	\subfloat[]{\includegraphics[width = \wid]{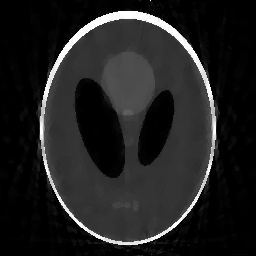}}
	
	\subfloat[]{\includegraphics[width = \wid]{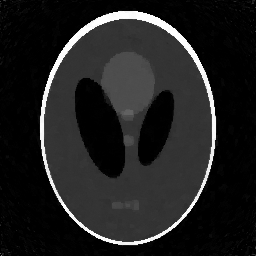}}
	\hspace{0.25in}
	\subfloat[]{\includegraphics[width = \wid]{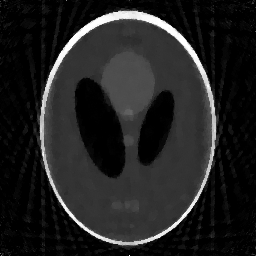}}
	
	\subfloat[]{\includegraphics[width = \wid]{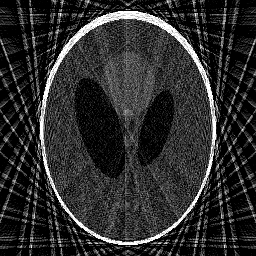}}
	\hspace{0.25in}
	\subfloat[]{\includegraphics[width = \wid]{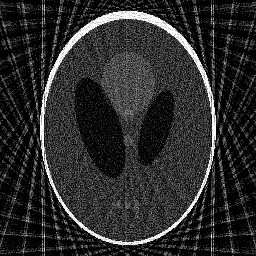}}
		
	\caption{A noise-free reconstruction with component-wise TVS method in (a) and the negative gradient TVS method in (b). A reconstruction with 2\% Gaussian noise with component-wise TVS method in (c) and the negative gradient TVS method in (d). FBP reconstructions are provided for noise-free and noisy reconstructions in (e) and (f), respectively. \label{fig: ART-TVS}}
\end{figure}

\begin{table}
	\caption{Simulated image reconstruction results. The stopping criterion differed between the noise-free and noisy data reconstructions and so results should not be compared between the two cases. Displayed values are averages of 30 trials and range is one standard deviation.\label{table: Results} }	
	\centering
	\renewcommand{\arraystretch}{1.5}
	\begin{tabular}{c|ccc|ccc}
		\multirow{2}{*}{Method} & \multicolumn{3}{c}{2\% Gaussian Noise} & \multicolumn{3}{c}{Noise-Free} 	 	
		\\ \cline{2-7} \
		& TV  & Time (s)  & Iterations & TV & {Time (s) } & Iterations  \\
		\hline
		CW-TVS & {$2032\pm11$} & {$40.0\pm0.1$} & {$106.9\pm0.6$} & {$1500\pm0$} & {$33.1\pm0.5$} & {$124\pm0$} \\
		\hline  
		NG-TVS &{$2941\pm897$} & {$38.3\pm10.3$} & {$25.0\pm7.5$} & {$1833\pm0$ } & {$143.5\pm1.2$ } & {$108\pm0$}
	\end{tabular}
\end{table}

\subsection{Computational results}

The image reconstruction results are shown in Table \ref{table: Results} and samples are visualized in Figure \ref{fig: ART-TVS}. Filtered back projection (FBP) images were also generated with AIR tools and are provided in Figure \ref{fig: ART-TVS} for reference to this traditional method using the noise-free and noisy data. Plots of TVS versus time and $\log(\|Ax^{k}-b\|)$ versus time are in shown in Figures \ref{fig: TV-plots-noise-free} and \ref{fig: Proximity-plots-noise-free}, respectively, for the noise-free data and Figures \ref{fig: TV-plots-noisy} and \ref{fig: Proximity-plots-noisy}, respectively, for the noisy data. Our computational example indicates a speedup with the CW-TVS method over the NG-TVS method in the noise-free case. As shown in Table \ref{table: Results}, for the noise-free experiment the TV output of the component-wise approach (1500) was noticeably superior to the negative gradient approach (1833), which required over 4 times more computation time. As seen in Figure \ref{fig: ART-TVS}a, the component-wise method yielded a faithful reconstruction with  negligible artifacts. The NG-TVS reconstruction with noise-free data had more blurred features (specifically around the white phantom border) and several artifacts outside the phantom. Note also that although the CW-TVS method was notably faster for the noise-free data, it required more iterations of ART than the NG-TVS method (124 versus 108). In Figure \ref{fig: TV-plots-noise-free}, we see the CW-TVS method TV function values appear to converge to the optimal TV value and at a much quicker rate than the NG-TVS method.

In the experiment with 2\% Gaussian noise added to the measurements, the component-wise approach still had superior TV output 	(2032 versus 2941). In the noisy-data case, the NG-TVS method was faster on average. For the CW-TVS reconstruction, artifacts may be seen in the corners of the image along with some blurring of edges, especially for the three small ellipses inside the phantom. The NG-TVS method displays more notable artifacts outside the phantom, faintly reminiscent of the streaks in the filtered back projection reconstructions. There are also blurring artifacts in the NG-TVS reconstruction. Lastly, in Figure \ref{fig: TV-plots-noisy}, we see the TV values increase over time.

\begin{figure*}
	\centering
	\begin{tikzpicture}
		\node[] at (0,0) {\includegraphics[width=2.5 in]{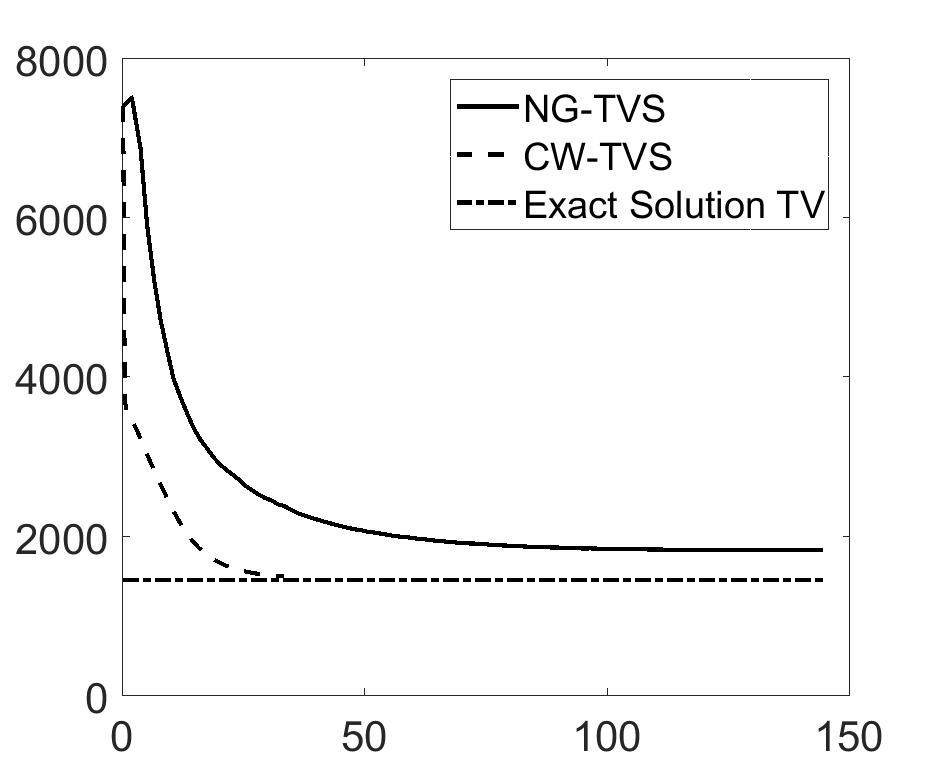}};
		\node[rotate = 90] at (-1.35 in,  0.1 in) {\large $\mbox{TV}(y^n)$};
		\node[]            at ( 0.1  in, -1.1 in) {\large Time (sec)};
	\end{tikzpicture}	
	\caption{Plots of TV from a noise-free data trial used to create the images in Subfigures \ref{fig: ART-TVS}A and \ref{fig: ART-TVS}B.
	\label{fig: TV-plots-noise-free}}
\end{figure*}

\begin{figure}
	\centering
	\begin{tikzpicture}
	\node[] at (0,0) {\includegraphics[width=2.5 in]{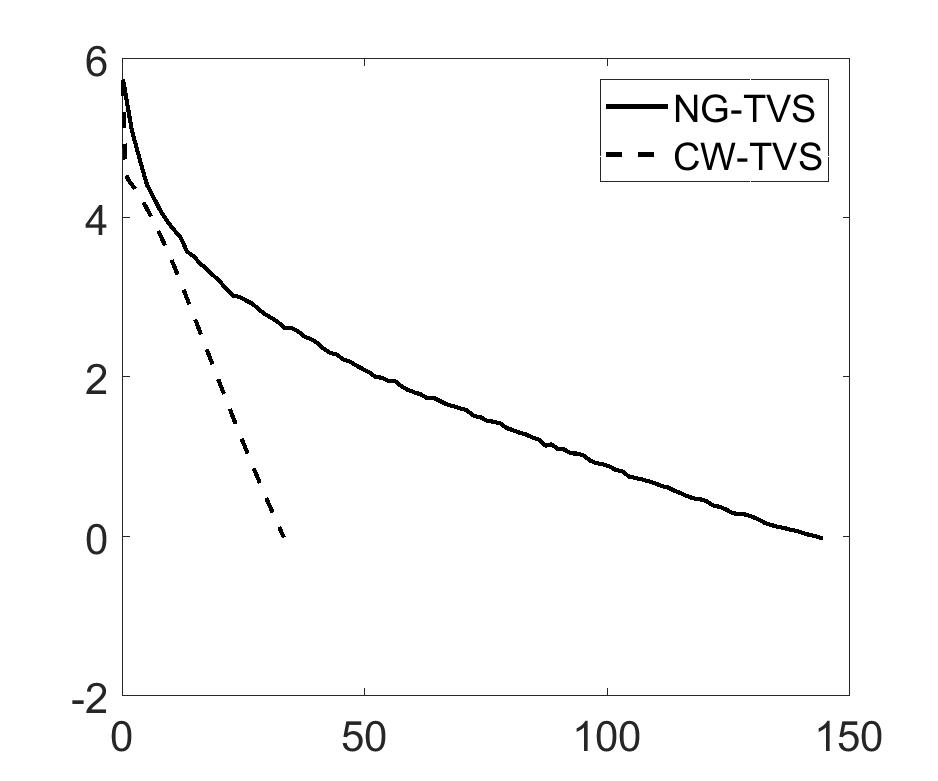}};
	\node[rotate = 90] at (-1.2 in,  0.1 in) {\large $\log(\|Ax-y^n\|)$};
	\node[]            at ( 0.1 in, -1.1 in) {\large Time (sec)};
	\end{tikzpicture}	
	\caption{Plots of proximity from a noise-free data trial used to create the images in Subfigures \ref{fig: ART-TVS}a and \ref{fig: ART-TVS}b.
	\label{fig: Proximity-plots-noise-free}}
\end{figure}

\begin{figure}
	\centering
	\begin{tikzpicture}
	\node[] at (0,0) {\includegraphics[width=2.5 in]{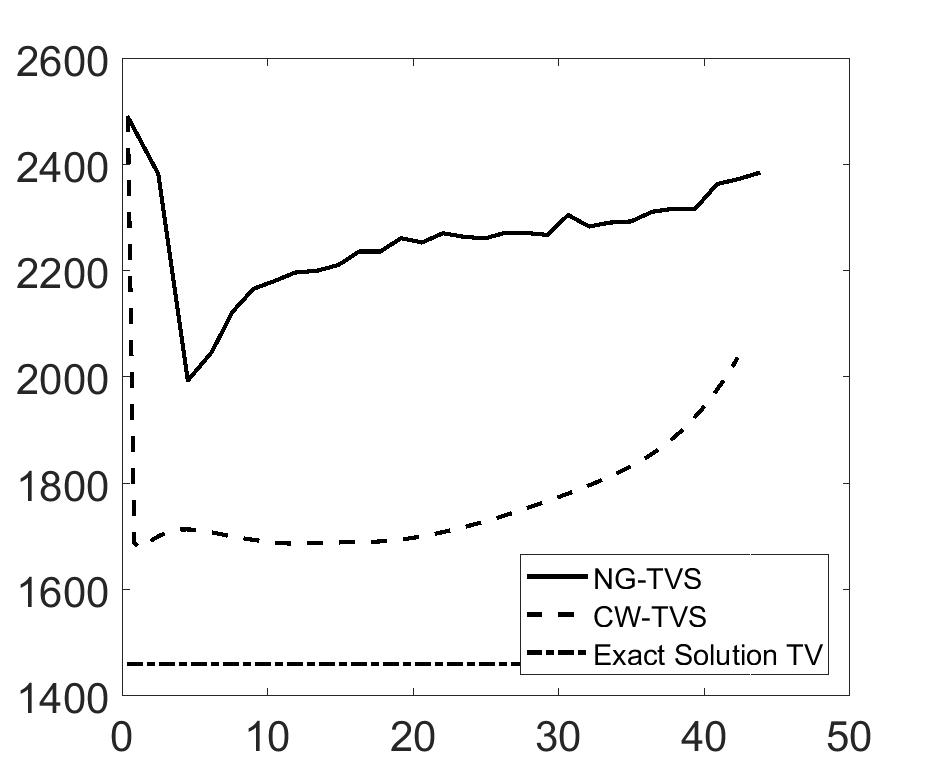}};
	\node[rotate = 90] at (-1.35 in,  0.1 in) {\large $\mbox{TV}(y^n)$};
	\node[]            at ( 0.1  in, -1.1 in) {\large Time (sec)};
	\end{tikzpicture}	
	\caption{Plots of TV from a noisy data trial used to create the images in Subfigures \ref{fig: ART-TVS}c and \ref{fig: ART-TVS}d.
	\label{fig: TV-plots-noisy}}
\end{figure}

\begin{figure}
	\centering
	\begin{tikzpicture}
	\node[] at (0,0) {\includegraphics[width=2.5 in]{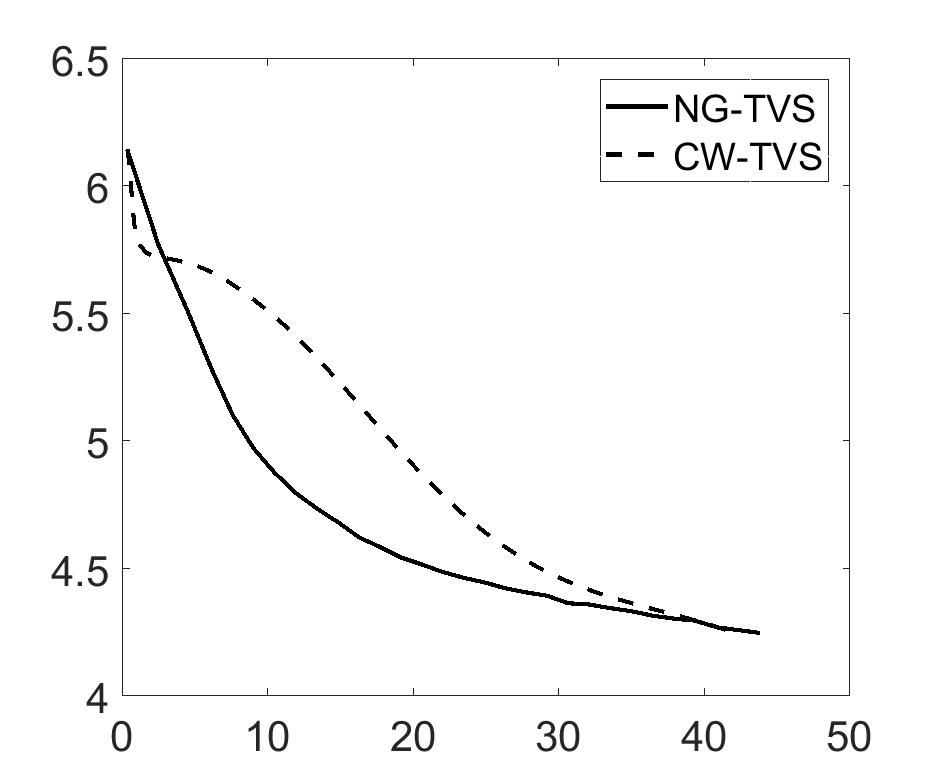}};
	\node[rotate = 90] at (-1.2 in,  0.1 in) {\large $\log(\|Ax-y^n\|)$};
	\node[]            at ( 0.1 in, -1.1 in) {\large Time (sec)};
	\end{tikzpicture}	
	\caption{Plots of proximity from a noisy data trial used to create the images in Subfigures \ref{fig: ART-TVS}c and \ref{fig: ART-TVS}d.
	\label{fig: Proximity-plots-noisy}}
\end{figure}

\begin{remark} The SM parameters such as $N_{k}$ and the number $a$ with which the parameters $\eta_\ell$ were generated, as well as the parameters associated with the feasibility-seeking ART were chosen as well as we could based on earlier published experiences and on some preliminary runs that we did with various values. The main point to observe in this regard is that they were identical in the runs with component-wise perturbations and the runs with negative subgradients as directions of nonasecent. Therefore, it is reasonable to assume that, since they were equal, they did not affect the comparative outcomes of the runs. Future methodological numerical investigations should address the choice of parameters systematically.
\end{remark}

\section{Discussion} \label{sec: Discussion}

\paragraph{\textbf {Relative computational costs}.}

Analysis of the difference in the computational cost of each TVS method is as follows. The negative gradient approach given in \ref{subsec:Target-Function-Reducing} requires the computation of up to three fractions for every component of the perturbation vector. In the denominator of each of these fractions is a square root term along with two multiplications and multiple additions/subtractions. On the other hand, the component-wise perturbation  method given in \ref{subsec:TVS-with-Coordinate-wise} requires only additions/subtractions, direct comparisons of floating point numbers with a minimum function, multiplication by the signum function, and division by 2. Hence we may expect the component-wise TVS method to be less computationally expensive per iteration. This is consistent with the results in Table \ref{table: Results}.

\paragraph{\textbf {Differences in output TV values}.}

	Inference of the difference in the output TV values from each method is as follows. With the CW-TVS method, the value of the perturbation to each pixel is bounded individually. This causes each component in a perturbation vector to be relatively equally weighted, thereby enabling a Gaussian distribution of entries. On the other hand, with NG-TVS the partial derivative is computed with respect to each pixel and then the vector is scaled as a whole. And, the partial derivative of TV at pixels along any sharp edge in the image has far greater magnitude than at the majority of pixels. This is the case with the piecewise-constant Shepp-Logan phantom. This implies that the distribution of values in the perturbation vector for NG-TVS in our reconstructions should consist mostly of near-zero values and a few  large nonzero values. Indeed, this is consistent with our observations. Also, due to these few pixels having large perturbations, it was observed in our reconstructions that the parameter $\ell$ often incremented several times between steps before perturbations had sufficiently small magnitude to reduce the TV. After several loops through the superiorization algorithm, this caused the $\eta_{\ell}$ bound to be so small that the perturbations to reduce TV became negligible. Thus, we assume that the chief advantage of the new TVS method with respect to TV value output is due to the component-wise bounding of perturbations.

\paragraph{\textbf {Connection to previous formulations of the SM}.}

Previous works (e.g., \cite{herman2012superiorization}) made note that convexity of the target function does not need to be assumed for superiorization. However, when the target function $\phi$ is not convex, there may exist a nonascent point $\eta_{\ell}d$ for which $d$ is not a nonascending vector. Additionally, if such a nonascent point is found, we do not need to be concerned whether there exists $\delta>0$ such that, for all $\lambda\in[0,\delta]$, the quantity $\lambda d$ gives a nonascent valued point. Hence the notion of a nonascending $\delta$-ball may be understood to be less restrictive and allow for a wider class of target function value reducing steps.

\section{Conclusions} \label{sec:Conclusions}

The superiorization methodology (SM) allows the conversion of a feasibility-seeking algorithm into a superiorized version of the  feasibility-seeking algorithm that, in addition to finding an $\varepsilon$-compatible solution of the constraints, steers iterates toward a reduced target function value. The superiorized version of the basic algorithm accomplishes this by interlacing target function nonascent steps into the original algorithm in an automatic fashion. This work has extended the scope of the SM by introducing the notion of nonascending $\delta$-balls for the nonascent steps. Using this notion, perturbation steps of the superiorized version of a basic algorithm can now be chosen from a wider class of target function value reducing steps, namely, functions that do not have any partial derivatives or whose partial derivatives cannot be calculated. Future investigations may also apply this formulation of the SM to problems where the target functions may not be convex (e.g., as often occurs in the field of intensity-modulated radiation therapy (IMRT) treatment planning) and to functions that are given only by tabular presentations.

We have presented an example that shows that our CW-TVS (component-wise total variation superiorization) method works well. As a byproduct we discovered that it finds a better solution than an NG-TVS (negative-gradient TVS) approach, and in less computation time in the noise-free case. Due to the limited scope of our numerical work we do not make any general claims. However, this finding is understandable in view of the simplicity of the component-wise perturbations with their averaging nature and the fact that the components of these perturbations are weighted relatively equally, which allows for a larger portion of an image to be smoothed with each perturbation than with the negative gradient approach. While the negative gradient approach directly attempts to reduce the target function of total variation, it is limited in its ability to remove artifacts from the image. We demonstrated this experimentally on a large-sized image reconstruction application that was modeled and set up as a constrained superiorization problem.
\newline
\newline
{\textbf{Acknowledgements}. We greatly appreciate the constructive comments of two anonymous reviewers which helped us improve the paper.}

\bibliographystyle{spmpsci}       
\bibliography{bib-der-free-sup} 

\begin{thebibliography}{10}
\providecommand{\url}[1]{{#1}}
\providecommand{\urlprefix}{URL }
\expandafter\ifx\csname urlstyle\endcsname\relax
  \providecommand{\doi}[1]{DOI~\discretionary{}{}{}#1}\else
  \providecommand{\doi}{DOI~\discretionary{}{}{}\begingroup
  \urlstyle{rm}\Url}\fi

\bibitem{Beck2009FastTVDenoising}
Beck, A., Teboulle, M.: Fast gradient-based algorithms for constrained total
  variation image denoising and deblurring problems.
\newblock IEEE Transaction on Image Processing \textbf{18}, 2419--2434 (2009)

\bibitem{bian2010evaluation}
Bian, J., Siewerdsen, J., Han, X., Sidky, E., Prince, J., Pelizzari, C., Pan,
  X.: Evaluation of sparse-view reconstruction from flat-panel-detector
  cone-beam {C}{T}.
\newblock Physics in Medicine and Biology \textbf{55}, 6575--6599 (2010)

\bibitem{butnariu2007stable}
Butnariu, D., Davidi, R., Herman, G.T., Kazantsev, I.G.: {Stable convergence
  behavior under summable perturbations of a class of projection methods for
  convex feasibility and optimization problems}.
\newblock IEEE Journal of Selected Topics in Signal Processing \textbf{1},
  540--547 (2007)

\bibitem{sup-bib}
Censor, Y.: Superiorization and perturbation resilience of algorithms: A
  bibliography compiled and continuously updated.
\newblock http://math.haifa.ac.il/yair/bib-superiorization-censor.html. See
  also: https://arxiv.org/abs/1506.04219.

\bibitem{censor2015weakstrong}
Censor, Y.: Weak and strong superiorization: Between feasibility-seeking and
  minimization.
\newblock Analele Stiintifice ale Universitatii Ovidius Constanta-Seria
  Matematica \textbf{23}, 41--54 (2015)

\bibitem{censorLinSup}
Censor, Y.: Can linear superiorization be useful for linear optimization
  problems?
\newblock Inverse Problems \textbf{33} (2017).
\newblock 044006.

\bibitem{censor2010perturbation}
Censor, Y., Davidi, R., Herman, G.T.: Perturbation resilience and
  superiorization of iterative algorithms.
\newblock Inverse Problems \textbf{26} (2010).
\newblock 065008.

\bibitem{censor2014projectedsubgrad}
Censor, Y., Davidi, R., Herman, G.T., Schulte, R.W., Tetruashvili, L.:
  Projected subgradient minimization versus superiorization.
\newblock Journal of Optimization Theory and Applications \textbf{160},
  730--747 (2014)

\bibitem{Sup-Special-Issue-2017}
Censor, Y., Herman, G.T., Jiang, M., (Editors): Superiorization: Theory and
  applications.
\newblock Journal of Inverse Problems \textbf{33}(4) (2017).
\newblock Special Issue.

\bibitem{strictfejer}
Censor, Y., Zaslavski, A.: Strict {F}ej{\'e}r monotonicity by superiorization
  of feasibility-seeking projection methods.
\newblock Journal of Optimization Theory and Applications \textbf{165},
  172--187 (2015)

\bibitem{Chan05recentdevelopments}
Chan, T., Esedoglu, S., Park, F., Yip, A.: Total variation image restoration:
  {O}verview and recent developments.
\newblock In: Handbook of Mathematical Models in Computer Vision, pp. 17--31.
  Springer Science+Business Media, Inc. (2006)

\bibitem{combettes2002adaptive}
Combettes, P., Luo, J.: An adaptive level set method for nondifferentiable
  constrained image recovery.
\newblock IEEE Transactions on Image Processing \textbf{11}, 1295--1304 (2002)

\bibitem{combettes2004image}
Combettes, P., Pesquet, J.C.: Image restoration subject to a total variation
  constraint.
\newblock IEEE Transactions on Image Processing \textbf{13}, 1213--1222 (2004)

\bibitem{davidi2009perturbation}
Davidi, R., Herman, G.T., Censor, Y.: Perturbation-resilient block-iterative
  projection methods with application to image reconstruction from projections.
\newblock International Transactions in Operational Research \textbf{16},
  505--524 (2009)

\bibitem{defrise2011algorithm}
Defrise, M., Vanhove, C., Liu, X.: An algorithm for total variation
  regularization in high-dimensional linear problems.
\newblock Inverse Problems \textbf{27} (2011).
\newblock 065002.

\bibitem{garduno2014}
Gardu\~{n}o, E., Herman, G.T.: Superiorization of the {ML-EM} algorithm.
\newblock Transactions on Nuclear Science \textbf{61}, 162--172 (2014)

\bibitem{garduno2017}
Gardu\~{n}o, E., Herman, G.T.: Computerized tomography with total variation and
  with shearlets.
\newblock Inverse Problems \textbf{33} (2017).
\newblock 044011.

\bibitem{garduno2011}
Gardu\~{n}o, E., Herman, G.T., Davidi, R.: Reconstruction from a few
  projections by $\ell_{1}$-minimization of the {H}aar transform.
\newblock Inverse Problems \textbf{27} (2011).
\newblock 055006.

\bibitem{gibalipetra2017}
Gibali, A., Petra, S.: D{C}-programming versus $\ell_{0}$-superiorization for
  discrete tomography.
\newblock Analele Stiintifice ale Universitatii Ovidius Constanta-Seria
  Matematica  (2017).
\newblock Accepted for publication. Available on ResearchGate.

\bibitem{Hansen2012airtools}
Hansen, P.C., Saxild-Hansen, M.: {AIR} {T}ools--{A} {MATLAB} package of
  algebraic iterative reconstruction methods.
\newblock Journal of Computational and Applied Mathematics \textbf{236},
  2167--2178 (2012)

\bibitem{neto2009incremental}
Helou~Neto, E., De~Pierro, {\'A}.: Incremental subgradients for constrained
  convex optimization: {A} unified framework and new methods.
\newblock SIAM Journal on Optimization \textbf{20}, 1547--1572 (2009)

\bibitem{neto2011perturbed}
Helou~Neto, E., De~Pierro, {\'A}.: On perturbed steepest descent methods with
  inexact line search for bilevel convex optimization.
\newblock Optimization \textbf{60}, 991--1008 (2011)

\bibitem{herman2009book}
Herman, G.T.: Fundamentals of Computerized Tomography, 2nd edn.
\newblock Springer-Verlag, London, UK (2009)

\bibitem{herman2012superiorization}
Herman, G.T., Gardu\~{n}o, E., Davidi, R., Censor, Y.: Superiorization: An
  optimization heuristic for medical physics.
\newblock Medical Physics \textbf{39}, 5532--5546 (2012)

\bibitem{Marquina2008image}
Marquina, A., Osher, S.: Image super-resolution by {TV}-regularization and
  {B}regman iteration.
\newblock Journal of Scientific Computing \textbf{37}, 367--382 (2008)

\bibitem{MATLAB}
MATLAB: A high-level language and interactive environment system by
  {M}athworks.
\newblock http://www.mathworks.com/products/matlab

\bibitem{Needell2013StableImage}
Needell, D., Ward, R.: Stable image reconstruction using total variation
  minimization.
\newblock SIAM Journal on Imaging Sciences \textbf{6}, 1035--1058 (2013)

\bibitem{nikazad2012accelerated}
Nikazad, T., Davidi, R., Herman, G.T.: Accelerated perturbation-resilient
  block-iterative projection methods with application to image reconstruction.
\newblock Inverse problems \textbf{28} (2012).
\newblock 035005.

\bibitem{nurminski2010envelope}
Nurminski, E.: Envelope stepsize control for iterative algorithms based on
  {F}ejer processes with attractants.
\newblock Optimization Methods and Software \textbf{25}, 97--108 (2010)

\bibitem{rios2013}
Rios, L., Sahinidis, N.: Derivative-free optimization: a review of algorithms
  and comparison of software implementations.
\newblock Journal of Global Optimization \textbf{56}, 1247--1293 (2013)

\bibitem{rudin1992nonlinear}
Rudin, L., Osher, S., Fatemi, E.: Nonlinear total variation based noise removal
  algorithms.
\newblock Physica D: Nonlinear Phenomena \textbf{60}, 259--268 (1992)

\bibitem{shen2002mathematical}
Shen, J., Chan, T.: Mathematical models for local nontexture inpaintings.
\newblock SIAM Journal on Applied Mathematics \textbf{62}, 1019--1043 (2002)

\bibitem{sidky2011constrained}
Sidky, E., Duchin, Y., Pan, X., Ullberg, C.: A constrained, total-variation
  minimization algorithm for low-intensity x-ray {CT}.
\newblock Medical Physics \textbf{38}, S117--S125 (2011)

\bibitem{sidky2008image}
Sidky, E., Pan, X.: Image reconstruction in circular cone-beam computed
  tomography by constrained, total-variation minimization.
\newblock Physics in Medicine and Biology \textbf{53}, 4777--4807 (2008)

\bibitem{wang2008new}
Wang, Y., Yang, J., Yin, W., Zhang, Y.: A new alternating minimization
  algorithm for total variation image reconstruction.
\newblock SIAM Journal on Imaging Sciences \textbf{1}, 248--272 (2008)

\bibitem{han2013image}
Zhang, H.M., Wang, L.Y., Yan, B., Li, L., Xi, X.Q., Lu, L.Z.: Image
  reconstruction based on total-variation minimization and alternating
  direction method in linear scan computed tomography.
\newblock Chinese Physics B \textbf{22} (2013).
\newblock 078701.

\end{thebibliography}

\end{document}